\documentclass[a4paper, draft, 12pt]{amsproc}
\usepackage{amssymb}
\usepackage{amsmath,amssymb,ulem,color}

\usepackage[hyphens]{url} 
\usepackage[dvips]{graphicx} 
\usepackage{cmdtrack} 
\usepackage{mathrsfs}                       

\theoremstyle{plain}
 \newtheorem{theorem}{Theorem}[section]

\theoremstyle{Definition}
 \newtheorem{exm}{Example}[section]
 
 \newtheorem{dfn}{Definition}[section]
\theoremstyle{remark}

 \numberwithin{equation}{section}

\renewcommand{\leq}{\leqslant}
\renewcommand{\geq}{\geqslant}

\title[Dynamics on the Wandering Components of the Fatou Set of Three \ldots]{Dynamics on the Wandering Components of the Fatou Set of Three Transcendental Entire Functions and Their Composites }



\author[B. H. Subedi]{\bfseries  Bishnu Hari Subedi}

\address{ 
Central Department of Mathematics \\ 
Institute of Science and Technology   \\ 
Tribhuvan University   \\ 
Kirtipur, Kathmandu\\
Nepal}
\email{subedi.abs@gmail.com / subedi\_bh@cdmathtu.edu.np }

\author[A. Singh]{Ajaya Singh}
\address{Central Department of Mathematics, Institute of Science and Technology, Tribhuvan University, Kirtipur, Kathmandu, Nepal }
\email{singh.ajaya1@gmail.com / singh\_a@cdmathtu.edu.np} 
\thanks{This research work of the first author is supported by PhD faculty fellowship from University Grants Commission, Nepal} 

\begin{document}

{\begin{flushleft}\baselineskip9pt\scriptsize
\end{flushleft}}
\vspace{18mm} \setcounter{page}{1} \thispagestyle{empty}

\begin{abstract}
We prove that there exist three transcendental entire functions that have infinite number of domains which lie in the wandering component of each of these functions and their composites. This result is a generalization of the result of Dinesh Kumar, Gopal Datt and Sanjay Kumar. In particular, they proved that there exist two transcendental entire functions that have infinite number of domains which lie in the wandering components of each of these functions and their composites

\end{abstract}

\maketitle
\section{Introduction}
We denote the \textit{complex plane} by $\mathbb{C}$, \textit{extended complex plane} by $\mathbb{C_{\infty}}$ and \textit{set of integers greater than zero} by $\mathbb{N}$. 
We assume the function $f:\mathbb{C}\rightarrow\mathbb{C}$ is \textit{transcendental entire function}  unless otherwise stated. 
For any $n\in\mathbb{N}, \;\; f^{n}$ always denotes the nth \textit{iterates} of $f$. 
If $f^{n}(z) = z$ for some smallest $n\in \mathbb{N}$, then we say that $z$ is periodic point of period n. In particular, if $f(z) = z$, then $z$ is a fixed point of $f$. 
If $| (f^{n}(z))^{\prime} | < 1$, where $ \prime $ represents complex differentiation of $f^{n}$ with respect to $z$, then $z$ is called attracting periodic point.  
A family $\mathscr{F} = \{f:  f\; \text{is meromorphic on some domain } \; X \; \text{of} \; \mathbb{C}_{\infty}\}$ forms normal family if every sequence $(f_{i})_{i\in \mathbb{N}}$ of functions contains a subsequence which converges uniformly to a finite limit or converges to $ \infty $ on every compact subset $D$ of $X$.

The \textit{Fatou set}\index{Fatou ! set} of $f$ denoted by $F(f)$ is the set of points $z\in \mathbb{C}$ such that sequence $(f^n)_{n\in \mathbb{N}}$ forms a normal family in some neighborhood of $z$. That is, $ z\in F(f) $ if $ z $ has a neighborhood $ U$ on which the family $\mathscr{F}$ is normal.  By definition, Fatou set is open and may or may not be empty.  Fatou set is non-empty for every entire function with attracting periodic points.

If $ U \subset F(f) $ (Fatou component), then $ f(U) $ lies in some component $ V $ of $ F(f) $ and $ V- f(U) $ is a set which contains at most one point (see for instance \cite{her}). Let $ U \subset F(f) $ (a Fatou component) such that $ f^{n}(U) $ for some $ n \in \mathbb{N} $, is contained in some component of $ F(f) $, which is usually denoted by $ U_{n} $.  A Fatou component $ U $ is called pre-periodic if there exist integers $ n >m \geq 0 $ such that $ U_{n} =  U_{m} $. In particular,  if $ U_{n} =  U_{0} = U$ ( that is, $ f^{n}(U) \subset U $) for some smallest positive integer $ n \geq 1 $, then $ U $ is called \textit{periodic Fatou component}\index{periodic! Fatou component} of period n and $\{ U_{0}, U_{1}\ldots, U_{n-1} \}$ is called the \textit{periodic cycle}\index{periodic ! cycle} of $ U $. A component of Fatou set $ F(f) $ which is not pre-periodic is called \textit{wandering domain}\index{wandering domain}. 

Our particular interest of this paper is that whether there are more than two transcendental entire functions that have similarity between the dynamics of their composites and dynamics of each of these functions. Dynamics of two transcendental entire functions and their composites were studied by A.P. Singh \cite{singh}. He constructed several examples of transcendental entire functions where dynamics of individual functions vary largely from the dynamics of their composites. In particular, A. P. Singh proved that there exists a domain which lies in the periodic component of individual functions and also lies in the periodic component of the one of the composite but lies in the wandering component of the other composites (Theorem 4).  Later, Dinesh Kumar, Gopal Datt and Sanjay Kumar \cite{kum8} extended this result to the possibility of having infinitely many domains satisfying the condition of A.P Singh’s result. In this paper, we investigate three transcendental entire functions such that each of individual functions as well as their every composite consists of infinite number of domains which lie in the wandering component of each of functions and their every composite. In particular, we prove the following result. 
\begin{theorem}\label{ne4}
There are transcendental entire functions $ f $, $ g $ and $ h $ such that there exist  infinite number of domains which lie in the wandering component of the $ F(f), F(g),  F(h),  F(f\circ g),  F(g\circ f),  F(f\circ h),  F(g\circ h),  F(h\circ f), F(h\circ g) , F(f\circ g \circ h), F(f\circ h \circ g), F(g\circ f \circ h), F(g\circ h \circ f), F(h\circ f \circ g)$ and $F(h\circ g \circ f)$. 
\end{theorem} 

\section{Carlemen Set}
To work out a proof of above theorem \ref{ne4}, first of all we need a notion of approximation theory of entire functions. In our case, we can use the notion of Carleman set from which we obtain approximation of any holomorphic map by entire functions
\begin{dfn}[\textbf{Carleman Set}]\label{cs}
Let $ F $ be a closed subset of $ \mathbb{C} $ and $ C(F) =\{f : F\to \mathbb{C}: f\; \text{is continuous on }\; S\; \text{and analytic in the interior of}\; F^{\circ}\; \text{of}\; F \}  $. Then $ F $ is called a Carleman set for $ \mathbb{C} $ if for any $ g \in C(F) $ and any positive continuous function $ \epsilon $ on $ F $, there exists entire function $ h $ such that $ |g(z) -h(z)| < \epsilon $ for all $ z \in F $. 
\end{dfn}
The following important characterization of Carleman set has been proved by A. Nersesjan in 1971 but we have been taken this result from \cite{gai}.
\begin{theorem}[{\cite[Theorem 4, page 157]{gai}}]\label{cs1}
Let $ F $ be proper subset of $ \mathbb{C} $. Then $ F  $ is a Carleman set for $ \mathbb{C} $ if and only if  $ F $ satisfies:
\begin{enumerate}
\item $ \mathbb{C}_{\infty} - F$ is connected;
\item $ \mathbb{C}_{\infty} - F$ is locally connected at $ \infty $;
\item for every compact subset $ K $ of  $ \mathbb{C} $, there is a neighborhood $ V $ of $ \infty $ in $ \mathbb{C}_{\infty}$ such that no component of $ F^{\circ} $ intersects both $ K $ and $ V $.
\end{enumerate}
\end{theorem}
It is well known in classical complex analysis that the space $ \mathbb{C}_{\infty} - F$ is connected if and only if each component $ Z $ of open set $ \mathbb{C} - F$ is unbounded. This fact together with above
 theorem \ref{cs1} will be a nice tool whether a  set  is a Carleman set for $ \mathbb{C} $. The sets given in the following examples are Carlemen sets for $ \mathbb{C} $.
 
\begin{exm}[{\cite[Example page 133]{gai}}]\label{cse}
The set $ E = \{z\in \mathbb{C} : |z| =1, \text{Re} z > 0 \} \cup \{z =x: x>1\} \cup \big(\bigcup_{n =3}^{\infty} \{z =r e^{i\theta}: r>1, \theta = \pi/n \}\big)$ is a Carleman set for $ \mathbb{C} $. 
\end{exm}
\begin{exm}[{\cite[Set S, page 131]{singh}}]\label{cs3}
The set $ E =G_{0} \cup\big( \bigcup_{k =1}^{\infty}(G_{k}\cup B_{K} \cup L_{k}\cup M_{k}) \big)$, where 
$$ 
G_{0} = \{z\in \mathbb{C}:|z-2| \leq 1\}; 
$$
 \begin{eqnarray*}
 G_{k} &= &\{z\in \mathbb{C} : |z -(4k +2)| \leq 1 \} \cup \{z \in \mathbb{C} : \text{Re}z = 4k +2,\;  \text{Im}z \geq 1 \}  \\ && \cup \{z \in \mathbb{C} : \text{Re}z = 4k +2, \;  \text{Im}z \leq -1 \},\; \; (k =1,2,3,\ldots);
\end{eqnarray*} 
 \begin{eqnarray*}
 B_{k} & = &\{z\in \mathbb{C} : |z +(4k +2)| \leq 1 \} \cup \{z \in \mathbb{C} : \text{Re}z = -(4k +2),\;  \text{Im}z \geq 1 \} \cup \\ && \{z \in \mathbb{C} : \text{Re}z = -(4k +2), \;  \text{Im}z \leq -1 \},\; \; (k =1,2,3,\ldots);
\end{eqnarray*}
 $$
 L_{k} = \{z \in \mathbb{C} : \text{Re}z = 4k \},\; \; (k =1,2,3,\ldots); 
 $$
 and 
 $$  
 M_{k} = \{z \in \mathbb{C} : \text{Re}z = -4k \},\; \; (k =1,2,3,\ldots)
 $$ 
is a Carleman set for $ \mathbb{C} $ by the theorem  \ref{cs1}. 
 \end{exm}
 
\section{Proof of the Main Result (Theorem 1)}
 If $f$ and $g$ are two transcendental entire functions, so are their composites $f\circ g$ and $g\circ f$ and the dynamics of one composite may help in the study of the dynamics of the other composite. In this regard, Bergweiler and Wang \cite{ber} proved the following results 
\begin{theorem}
Let $f$ and $g$ be two transcendental entire functions. Then $f\circ g$ has wandering domains if and only if $g\circ f$ has wandering domains. 
\end{theorem}
It is also note that certain classes of entire functions do not have wandering domains (see for instance theorem 3 of Bergweiler and Wang \cite{ber}). It is known that the dynamics of $f\circ g$ are very similar to the dynamics of $g\circ f$. A. P. Singh \cite{singh} interested to know whether there is similarity between dynamics of individual entire functions and their composites. However, in reality, it does not hold in general. From the help of the Carleman set of example 2, A.P. Singh \cite{singh} proved the following result which shows that certain amount of similarity might be hold.
From the help of the Carleman set of the above example \ref{cs3}, A.P. Singh {\cite[Theorem 2]{singh}} proved the following result.

\begin{theorem}\label{ne3}
There are transcendental entire functions $ f $ and $ g $ such that there exists a domain which lies in the wandering component of the $ F(f),\; F(g), \; F(f\circ g) $ and $ F(g\circ f) $.
\end{theorem}

In fact, A. P. Singh \cite{singh} also had proved other results regarding the dynamics of two individual functions and their composites (see for instance {\cite[Theorem 1, Theorem 3 and Theorem 4] {singh}}) which are also stricly based on the Carleman set of example \ref{cs3}. Dinesh Kumar, Gopal Datt and Sanjay Kumar  extended  these result of A.P. Singh in {\cite[Theorem 2.1 to Theorem 2.15]{kum8}}. For our purpose, we cite the following two results from \cite{kum8}. 
\begin{theorem} [\textbf{\cite[Theorem 2.13]{kum8}}]\label{ne1}
There are transcendental entire functions $ f $ and $ g $ such that there exist  infinite number of domains which lie in the pre-periodic component of the $ F(f), F(g)$, $F(f\circ g) $ and $ F(g\circ f) $.
\end{theorem}
\begin{theorem}[\textbf{{\cite[Theorem 2.2]{kum8}}}] \label{ne2}
There are transcendental entire functions $ f $ and $ g $ such that there exists infinite number of domains which lie in the wandering component of the $ F(f),  F(g)$, $ F(f\circ g) $ and $ F(g\circ f) $.
\end{theorem}
Our main result, that is, theorem \ref{ne4} is an extension of the theorem \ref{ne2}. As stated in the theorem \ref{ne4}, theorem \ref{ne2} can be extended to the existence of more than two transcendental entire functions such that each individual functions and their composites may have infinitely many domains which lie in the wandering component of each of the functions and their composites. We proceed the following long proof of the theorem\ref{ne4}. 
\begin{proof}[Proof of the Theorem \ref{ne4}]
The set of the form 
$$ 
E =G_{0} \cup \big(\bigcup_{k =1}^{\infty}(G_{k}\cup B_{K} \cup L_{k}\cup M_{k})\big).
$$
 where $ G_{0}, G_{k}, B_{k}, L_{k}$ and $ M_{k} $ are sets as defined in above example \ref{cs3}, is a Carleman set for $ \mathbb{C} $. 
By the continuity of exponential map, we can write, for a given $ \epsilon > 0 $, there exists $ \delta > 0 $ such that 
$$
|w -w_{0}| < \delta \Longrightarrow |e^{w} - e^{w_{0}}| < \epsilon.
$$ 
Let us choose $ \epsilon = 1/2 $, then  there exist sufficiently small $ \delta _{k} > 0, \; \delta _{k}^{'}>0$ and $\delta _{k}^{''}>0 $  such that
$$
|w-(\pi i+ \log(4k +6))| <\delta _{k}
\Longrightarrow  |e^{w} +(4k +6)| < 1/2, \; (k =1,2,3,\ldots); 
$$
$$ 
 |w-\log(4k - 2)| <\delta _{k}^{'}
\Longrightarrow  |e^{w}- (4k +6)| < 1/2, \; (k =1,2,3,\ldots); 
 $$
 and
 $$
|w-\log(4k -6)| <\delta _{k}^{''}
\Longrightarrow  |e^{w} +(4k -6)| < 1/2, \; (k =3,4,5,\ldots); 
 $$
 In particular, let us choose sufficiently small $ \delta_{0} > 0, \;   \delta_{1}^{'} > 0 $ and  $ \delta_{2}^{'} > 0 $ such that 
$$
|w -\log 2| < \delta_{0} \Longrightarrow |e^{w} -2| < 1/2.
$$
$$
|w -(\pi i+\log 6)| < \delta_{1}^{'} \Longrightarrow |e^{w} +6| < 1/2.
$$
and
$$
|w -(\pi i+\log 10)| < \delta_{2}^{'} \Longrightarrow |e^{w} +10| < 1/2.
$$
Next, let us define the following functions:
$$
\alpha(z) =
\left\{
\begin{array}{lll}
\log 2, & \forall z \in G_{0} \cup \big(\bigcup_{k =1}^{\infty} (L_{k}\cup  M_{k})\big);\\
\pi i +\log 6, & \forall z \in G_{k}, \; k =1,2,3,\ldots;\\
\pi i +\log(4k +6), & \forall z \in B_{k}, \; k =1,2,3,\ldots;
\end{array}
\right\}
$$

$$
\beta(z) =
\left\{
\begin{array}{lll}
\log 2, & \forall z \in G_{0} \cup \big(\bigcup_{k =1}^{\infty} (L_{k}\cup  M_{k})\big);\\
\log(4k -2), & \forall z \in G_{k}, \; k =2,3,4,\ldots;\\
\pi i +\log 6, & \forall z \in G_{1};\\
\pi i +\log(4k +6), & \forall z \in B_{k}, \; k =1,2,3,\ldots;\\
\end{array}
\right\}
$$

$$
\gamma(z) =
\left\{
\begin{array}{lll}
\log 2, & \forall z \in G_{0} \cup \big(\bigcup_{k =1}^{\infty} (L_{k}\cup  M_{k})\big);\\
\log(4k -6), & \forall z \in G_{k}, \; k =3,4,5\ldots;\\
\pi i +\log 6, & \forall z \in G_{2}; \\
\pi i +\log 10, & \forall z \in G_{1}; \\
\pi i + \log(4k +6), & \forall z \in B_{k}, \; k =1,2,3,\ldots;\\

\end{array}
\right\}
$$
Let us define again the following functions:
$$
\epsilon_{1}(z) =
\left\{
\begin{array}{lll}
\delta_{0}, & \forall z \in G_{0} \cup \big(\bigcup_{k =1}^{\infty} (L_{k}\cup  M_{k})\big);\\
\delta _{1}^{'}, & \forall z \in G_{k}, \; k =1,2,3,\ldots;\\
\delta _{k}, & \forall z \in B_{k}, \; k =1,2,3,\ldots;
\end{array}
\right\}
$$

$$
\epsilon_{2}(z) =
\left\{
\begin{array}{lll}
\delta_{0}, & \forall z \in G_{0} \cup \big(\bigcup_{k =1}^{\infty} (L_{k}\cup  M_{k})\big);\\
\delta_{k}^{'}, & \forall z \in G_{k}, \; k =2,3,4\ldots;\\
\delta_{1}^{'}, & \forall z \in G_{1};\\
\delta_{k}, &\forall z \in B_{k}, \; k =1,2,3,\ldots;\\

\end{array}
\right\}
$$
and
$$
\epsilon_{3}(z) =
\left\{
\begin{array}{lll}
\delta_{0}, & \forall z \in G_{0} \cup \big(\bigcup_{k =1}^{\infty} (L_{k}\cup  M_{k})\big);\\
\delta_{k}^{''}, & \forall z  \in G_{k}, \; k = 3,4,5\ldots;\\
\delta_{1}^{'}, & \forall z \in G_{2}; \\
\delta_{2}^{'}, & \forall z \in G_{1}; \\
\delta_{k}^{'}, & \forall z \in B_{k}, \; k =1,2,3,\ldots; \\

\end{array}
\right\}
$$
Clearly, the functions $ \alpha(z), \; \beta(z) $ and $ \gamma (z) $ are piece wise constant functions, so they are continuous on the set $ E $ and analytic in $ E^{\circ} $. Also,
since $ E $ is a Carleman set, so there exist entire functions $ f_{1}(z), \;  g_{1}(z)$ and $ h_{1}(z) $ such that 
$$
\forall z \in E, \;\; |f_{1}(z) -\alpha(z)| \leq \epsilon_{1}(z),\;|g_{1}(z) -\beta(z)| \leq \epsilon_{2}(z)\;  \text{and}\; |h_{1}(z) -\gamma(z)| \leq \epsilon_{3}(z).
$$
Consequently, we get transcendental entire functions $ f(z) = e^{f_{1}(z)},\; g(z) = e^{g_{1}(z)} $  and $ h(z) = e^{h_{1}(z)} $ which respectively satisfy the following:	
\begin{equation}\label{cs19}
\left\{
\begin{array}{lll}
|f(z) -2|< 1/2,\; \; \;\;  &  \forall z \in G_{0} \cup\big(\bigcup_{k =1}^{\infty} (L_{k}\cup  M_{k})\big);\\
|f(z) +6|< 1/2,\;\;\; &  \forall z \in G_{k}, \; k =1,2,3\ldots;\\ 
|f(z) +(4k+6)|< 1/2, \;\; \;&   \forall z \in B_{k}, \; k =1,2,3,\ldots; 
\end{array}
\right\}
\end{equation}
		
\begin{equation}\label{cs20}
\left\{
\begin{array}{lll}
|g(z) -2|< 1/2,  \;\;\;& \forall z \in G_{0} \cup \big( \bigcup_{k =1}^{\infty} (L_{k}\cup  M_{k})\big);\\
|g(z) -(4k-2)|< 1/2, \;\;\;&   \forall z \in G_{k}, \; k =2,3,4,\ldots;\\
 |g(z) +6|< 1/2,\;\;\; & \forall z \in G_{1};\\
|g(z) +(4k+6)|< 1/2, \;\;\; &  \forall z \in B_{k}, \; k =1,2,3,\ldots; \\
 
\end{array}
\right\}
\end{equation}	
and 
\begin{equation}\label{cs21}
\left\{
\begin{array}{lll}
|h(z) -2|< 1/2, \;\;\; &  \forall z \in G_{0} \cup \bigcup_{k =1}^{\infty} (L_{k}\cup  M_{k});\\
|h(z) - (4k-6)|< 1/2, \;\;  & \forall z \in G_{k}, k =3,4,5\ldots;\\
|h(z) +6|<1/2, \;\; & \forall z \in G_{2};\\
|h(z) +10|<1/2, \;\; & \forall z \in G_{1};\\

|h(z) +(4k+6)|< 1/2, \;\;\; &  \forall z \in B_{k}, k =1,2,3\ldots;\\
\end{array}
\right\}
\end{equation}
As we did just above in \ref{cs19}, \ref{cs20}, and \ref{cs21}, each of the functions $ f $, $ g $ and $h$ maps the domain $ G_{0} \cup \bigcup_{k =1}^{\infty} (L_{k}\cup  M_{k}) $ into smaller disk $ |z -2|< 1/2 $ contained in $ G_{0} $ and each of these function is a contracting mapping. So, $ G_{0} \cup \big(\bigcup_{k =1}^{\infty} (L_{k}\cup M_{k})\big) $ contains a fixed point $ z_{0}$ (say) such that 
$$
f^{n}(G_{0} \cup \bigcup_{k =1}^{\infty} (L_{k}\cup  M_{k})) \longrightarrow z_{0} \; \text{as}\; n \longrightarrow \infty.
$$
$$
g^{n}(G_{0} \cup \bigcup_{k =1}^{\infty} (L_{k}\cup  M_{k})) \longrightarrow z_{0} \; \text{as}\; n \longrightarrow \infty.
$$
$$
h^{n}(G_{0} \cup \bigcup_{k =1}^{\infty} (L_{k}\cup  M_{k})) \longrightarrow z_{0} \; \text{as}\; n \longrightarrow \infty.
$$
This fixed point $ z_{0} $ is attracting fixed point for each function $ f $, $ g $ and $h$, so $ G_{0} \cup \big(\bigcup_{k =1}^{\infty} (L_{k}\cup  M_{k})) $ lies in attracting cycle and hence $ G_{0} \cup\big( \bigcup_{k =1}^{\infty} (L_{k}\cup M_{k})) $ is a subset of each of the Fatou set $ F(f), F(g) $ and $ F(h) $. In this case, $ J(f) \neq \mathbb{C}, \; J(g) \neq \mathbb{C} $ and $ J(h) \neq \mathbb{C} $ and so Julia set of each of the function $ f $, $ g $ and $h$ does not contain interior point and hence Fatou set of each of these function contains all interior points. In such case, Fatou set of each of the function $ f $, $ g $ and $h$ contains Carleman set $ E $.

Again, as defined in above equation \ref{cs19}, function $ f $ maps each $ G_{k} $ into smaller disk contained in $ B_{1} $ and each $ B_{k} $ into smaller disk contained in $ B_{k +1} $. In fact, $ G_{k} $ and $ B_{k} $ are contained in the wandering components of Fatou set $ F(f) $ of the function $ f $. Also, as defined in equation \ref{cs20}, function $ g $ maps each of the domains $ G_{k} $ into the smaller disk contained in $ G_{k -1},\ (k =2,3,4,\ldots)$,  $ G_{1} $ into smaller disk contained in $ B_{1} $ and  $ B_{k}, \; (k =1,2,3, \ldots) $ into the smaller disks contained in $ B_{k+1} $. In fact, $ G_{k}$ and $ B_{k} $ are contained in the wandering components of the Fatou set $ F(g) $ of the function $ g $. Likewise, as defined in equation  \ref{cs21}, domains $ G_{k} $ and $ B_{k}, (k =1,2,3 \ldots) $ are contained in the wandering components  under the function $ h $. 

Next, we examine the dynamical behavior of composites of the functions $ f $, $ g $ and $h$. The composite of any two and all of three of these functions satisfy the following:\\

\underline{Dynamical behavior of $f \circ g$}:\\
\begin{equation} \label{cs22}
\left\{
\begin{array}{lll}
|(f \circ g)(z) -2|< 1/2,  \;\;\;& \forall z \in G_{0} \cup \big(\bigcup_{k =1}^{\infty} (L_{k}\cup  M_{k})\big);\\
|(f \circ g)(z) + 6|< 1/2, \;\;\;&   \forall z \in G_{k}, \; k =2,3,4\ldots;\\
 |(f\circ g)(z) + 10|< 1/2, \;\;\; &   \forall z \in G_{1};  \\
|(f \circ g)(z) -(4k+10)|< 1/2, \;\;\; &  \forall z \in B_{k}, \; k =1,2,3,\ldots; \\

\end{array}
\right\}
\end{equation}
This composition rule \ref{cs22} shows that the domains $G_{0} \cup \big(\bigcup_{k =1}^{\infty} (L_{k}\cup  M_{k})\big)  $, $ G_{k} $ and $ B_{k}, \; (k =1,2,3,\ldots) $ belong to $ F(f \circ g) $ and in fact, each $ G_{k} $ and $ B_{k}$ is contained in the wandering   components of  $ F(f \circ g) $. \\

\underline{Dynamical behavior of  $g \circ f$}:\\	
\begin{equation} \label{cs23}
\left\{
\begin{array}{lll}
|(g \circ f)(z) -2|< 1/2,  \;\;\;& \forall z \in G_{0} \cup \big(\bigcup_{k =1}^{\infty} (L_{k}\cup  M_{k})\big);\\
|(g \circ f)(z) + 10|< 1/2, \;\;\;&   \forall z \in G_{k}, \; k =1, 2,3,\ldots;\\
|(g \circ f)(z) + (4k+10)|< 1/2, \;\;\; &  \forall z \in B_{k}, \; k =1,2,3,\ldots; \\

\end{array}
\right\}
\end{equation}
From this composition rule \ref{cs23}, we can say that the domains $G_{0} \cup \big(\bigcup_{k =1}^{\infty} (L_{k}\cup  M_{k})\big)$, $ G_{k} $ and $ B_{k}, \; (k =1,2,3,\ldots) $ belong to $ F(g \circ f) $ and in fact, each $ G_{k} $ and $ B_{k}$ belongs to the wandering component of $ F(g \circ f) $.\\

\underline{Dynamical behavior of $f \circ h$}:\\	
\begin{equation} \label{cs24}
\left\{
\begin{array}{lll}
|(f \circ h)(z) -2|< 1/2,  \;\;\;& \forall z \in G_{0} \cup \big(\bigcup_{k =1}^{\infty} (L_{k}\cup  M_{k})\big);\\
|(f \circ h)(z) + 14|< 1/2, \;\;\;&   \forall z \in G_{1};\\
|(f \circ h)(z) + 10|< 1/2, \;\;\;&   \forall z \in G_{2};\\
 |(f \circ h)(z) + 6|< 1/2, \;\;\;&   \forall z \in G_{k}, \; k =4, 5, 6,\ldots;\\
 |(f \circ h)(z) + (4k+10)|< 1/2, \;\;\; &  \forall z \in B_{k}, \; k =1,2,3,\ldots; \\

\end{array}
\right\}
\end{equation}
As defined in the above composition rule \ref{cs24}, the domains $G_{0} \cup \big(\bigcup_{k =1}^{\infty} (L_{k}\cup  M_{k})\big)$, $ G_{k} $ and $ B_{k}, \; (k =1,2,3,\ldots) $ belong to $ F(f \circ h) $ and in fact, each $ G_{k} $ and $ B_{k}$ for all $k =1,2,3, \ldots  $ belongs to the wandering  components of  $ F(f \circ h) $. \\
 
\underline{Dynamical behavior of  $h \circ f$}:\\	
\begin{equation} \label{cs25}
\left\{
\begin{array}{lll}
|(h \circ f)(z) -2|< 1/2,  \;\;\;& \forall z \in G_{0} \cup \big(\bigcup_{k =1}^{\infty} (L_{k}\cup  M_{k})\big);\\
|(h \circ f)(z) + 10|< 1/2, \;\;\;&   \forall z \in G_{k}, \; k =1, 2, 3,\ldots;\\
 |(h \circ f)(z) + (4k+10)|< 1/2, \;\;\; &  \forall z \in B_{k}, \; k =1,2,3,\ldots; \\

\end{array}
\right\}
\end{equation}
From this composition rule \ref{cs25}, we can say that the domains $G_{0} \cup \big(\bigcup_{k =1}^{\infty} (L_{k}\cup  M_{k})\big)$, $ G_{k} $ and $ B_{k}$ for all $k =1,2,3,\ldots $ belong to $ F(h \circ f) $ and in fact, each $ G_{k} $ and $ B_{k}$  for all $k =1,2,3,\ldots $ is contained in the wandering components of  $ F(h \circ f) $.\\ 
 
\underline{Dynamical behavior of $g \circ h$}:\\	
\begin{equation} \label{cs26}
\left\{
\begin{array}{lll}
|(g \circ h)(z) -2|< 1/2,  \;\;\;& \forall z \in G_{0} \cup \big(\bigcup_{k =1}^{\infty} (L_{k}\cup  M_{k})\big);\\
|(g\circ h)(z) + 14|< 1/2, \;\;\; &   \forall z \in G_{1}; \\
|(g\circ h)(z) + 10|< 1/2, \;\;\; &   \forall z \in G_{2}; 
\\
|(g\circ h)(z) + 6|< 1/2, \;\;\; &   \forall z \in G_{k}; \; k =3, 4, 5 \ldots; \\
|(g \circ h)(z) + (4k+10)|< 1/2, \;\;\; &  \forall z \in B_{k}, \; k =1,2,3,\ldots; \\

\end{array}
\right\}
\end{equation}
As defined in the above composition rule \ref{cs26}, the domains $G_{0} \cup \big(\bigcup_{k =1}^{\infty} (L_{k}\cup  M_{k})\big)$, $ G_{k} $ and $ B_{k}, \; (k =1,2,3,\ldots) $ belong to $ F(g \circ h) $ and in fact, each $ G_{k} $ and $ B_{k}$ for all $ k =1,2,3, \ldots $ is contained in wandering components of  $ F(g \circ h) $. \\

\underline{Dynamical behavior of $h \circ g$}:\\	
\begin{equation} \label{cs27}
\left\{
\begin{array}{lll}
|(h \circ g)(z) -2|< 1/2,  \;\;\;& \forall z \in G_{0} \cup \big(\bigcup_{k =1}^{\infty} (L_{k}\cup  M_{k})\big);\\
|(h\circ g)(z) + 10|< 1/2, \;\;\; &   \forall z \in G_{k} , k =1,2; \\
|(h\circ g)(z) + 6|< 1/2, \;\;\; &   \forall z \in G_{3} , ; \\
|(h \circ g)(z) -(4k-10)|< 1/2, \;\;\;&   \forall z \in G_{k}, \; k =4, 5, 6,\ldots;\\
 |(h \circ g)(z) + (4k+10)|< 1/2, \;\;\; &  \forall z \in B_{k}, \; k =1,2,3,\ldots; \\

\end{array}
\right\}
\end{equation}
As defined in the above composition rule \ref{cs27}, the domains $G_{0} \cup \big(\bigcup_{k =1}^{\infty} (L_{k}\cup  M_{k})\big)  $, $ G_{k} $ and $ B_{k}, \; (k =1,2,3,\ldots) $ belong to $ F(h \circ g) $ and in fact, each $ G_{k} $ and $ B_{k}$ for all $ k =1,2,3, \ldots $ is contained in wandering components of  $ F(h \circ g) $. \\

\underline{Dynamical behavior of $f \circ g \circ h$}:\\	
\begin{equation} \label{cs28}
\left\{
\begin{array}{lll}
|(f \circ g \circ h)(z) -2|< 1/2,  \;\;\;& \forall z \in G_{0} \cup \big(\bigcup_{k =1}^{\infty} (L_{k}\cup  M_{k})\big);\\
|(f\circ g\circ h)(z) +18|< 1/2, \;\;\; &   \forall z \in G_{1}; \\
|(f\circ g\circ h)(z) +14|< 1/2, \;\;\; &   \forall z \in G_{2}; \\
|(f\circ g\circ h)(z) +10|< 1/2, \;\;\; &   \forall z \in G_{3}; \\
|(f \circ g \circ h)(z) + 6|< 1/2, \;\;\;&   \forall z \in G_{k}, \; k =4, 5, 6,\ldots;\\
|(f \circ g\circ  h)(z) + (4k +10)|< 1/2, \;\;\;&   \forall z \in B_{k}, \; k =1,2,3, \ldots; \\

\end{array}
\right\}
\end{equation}
The composition rule \ref{cs28} assigned above tells us  that domains $G_{0} \cup \big(\bigcup_{k =1}^{\infty} (L_{k}\cup  M_{k})\big)  $, $ G_{k} $ and $ B_{k}, \; (k =1,2,3,\ldots) $ lie in $ F(f \circ g \circ h) $ and in fact, each $ G_{k} $ and $ B_{k}$ for all $ k =1,2,3, \dots  $ is contained in the wandering component of  $ F(f \circ g \circ h) $. \\

\underline{Dynamical behavior of $f \circ h \circ g$}:\\	
\begin{equation} \label{cs29}
\left\{
\begin{array}{lll}
|(f \circ h \circ g)(z) -2|< 1/2,  \;\;\;& \forall z \in G_{0} \cup \big(\bigcup_{k =1}^{\infty} (L_{k}\cup  M_{k})\big);\\
|(f\circ h\circ g)(z) +14|< 1/2, \;\;\; &   \forall z \in G_{k},\; \text{for} \; k =1,2; \\
|(f\circ h\circ g)(z) +10|< 1/2, \;\;\; &   \forall z \in G_{3},\; \text{for} \; k =1,2; \\
|(f \circ h \circ g)(z) + 6|< 1/2, \;\;\;&   \forall z \in G_{k}, \; k =4, 5, 6,\ldots;\\
 |(f \circ h\circ  g)(z) + (4k +10)|< 1/2, \;\;\;&   \forall z \in B_{k}, \; k =1,2,3, \ldots; \\
\end{array}
\right\}
\end{equation}
The composition rule \ref{cs29} assigned above tells us  that domains $G_{0} \cup \big(\bigcup_{k =1}^{\infty} (L_{k}\cup  M_{k}) \big) $, $ G_{k} $ and $ B_{k}, \; (k =1,2,3,\ldots) $ lie in $ F(f \circ g \circ h) $ and in fact, each $ G_{k} $ and $ B_{k}$ for all $ k =1,2,3, \dots  $ is contained in the wandering component of  $ F(f \circ h \circ g) $. \\

\underline{Dynamical behavior of $g \circ f \circ h$}:\\	
\begin{equation} \label{cs30}
\left\{
\begin{array}{lll}
|(g \circ f \circ h)(z) -2|< 1/2,  \;\;\;& \forall z \in G_{0} \cup \big(\bigcup_{k =1}^{\infty} (L_{k}\cup  M_{k})\big);\\
|(g \circ f \circ h)(z) + 18|< 1/2, \;\;\;&   \forall z \in G_{1}; \\
|(g \circ f \circ h)(z) + 14|< 1/2, \;\;\;&   \forall z \in G_{2};\\
|(g \circ f \circ h)(z) + 10|< 1/2, \;\;\;&   \forall z \in G_{k}, \; k =3, 4, 5, \ldots;\\
 |(g \circ f\circ  h)(z) + (4k +10)|< 1/2, \;\;\;&   \forall z \in B_{k}, \; k =1,2,3, \ldots; \\

\end{array}
\right\}
\end{equation}
The composition rule \ref{cs30} assigned above tells us  that domains $G_{0} \cup \big(\bigcup_{k =1}^{\infty} (L_{k}\cup  M_{k})\big) $, $ G_{k} $ and $ B_{k}, \; (k =1,2,3,\ldots) $ lie in $ F(g \circ f \circ h) $ and in fact, each $ G_{k} $ and $ B_{k}$ for all $ k =1,2,3, \dots  $ is contained in the wandering component of  $ F(g \circ f \circ h) $. \\

\underline{Dynamical behavior of $g \circ h \circ f$}:\\	
\begin{equation} \label{cs31}
\left\{
\begin{array}{lll}
|(g \circ h \circ f)(z) -2|< 1/2,  \;\;\;& \forall z \in G_{0} \cup \big(\bigcup_{k =1}^{\infty} (L_{k}\cup  M_{k})\big);\\
|(g \circ h \circ f)(z) + 14|< 1/2, \;\;\;&   \forall z \in G_{k}, \; k = 1, 2, 3, \ldots;\\
 |(g \circ h\circ  f)(z) + (4k +10)|< 1/2, \;\;\;&   \forall z \in B_{k}, \; k =1,2,3,\ldots; \\

\end{array}
\right\}
\end{equation}
The composition rule \ref{cs31} assigned above tells us  that domains $G_{0} \cup \big(\bigcup_{k =1}^{\infty} (L_{k}\cup  M_{k})\big)  $, $ G_{k} $ and $ B_{k}, \; (k =1,2,3,\ldots) $ lie in $ F(g \circ h \circ f) $ and in fact, each $ G_{k} $ and $ B_{k}$ for all $ k =1,2,3, \dots  $ is contained in the wandering component of  $ F(g \circ h \circ f) $. \\

\underline{Dynamical behavior of $h \circ f \circ g$}:\\	
\begin{equation} \label{cs32}
\left\{
\begin{array}{lll}
|(h \circ f \circ g)(z) -2|< 1/2,  \;\;\;& \forall z \in G_{0} \cup \big(\bigcup_{k =1}^{\infty} (L_{k}\cup  M_{k})\big);\\
|(h \circ f \circ g)(z) + 14|< 1/2, \;\;\;&   \forall z \in G_{1};\\
|(h \circ f \circ g)(z) + 10|< 1/2, \;\;\;&   \forall z \in G_{k}, \; k =2, 3, 4,\ldots;\\

|(h \circ f\circ  g)(z) + (4k +10)|< 1/2, \;\;\;&   \forall z \in B_{k}, \; k =1,2,3, \ldots; \\

\end{array}
\right\}
\end{equation}
The composition rule \ref{cs32} assigned above tells us  that domains $G_{0} \cup \big(\bigcup_{k =1}^{\infty} (L_{k}\cup  M_{k})\big)  $, $ G_{k} $ and $ B_{k}, \; (k =1,2,3,\ldots) $ lie in $ F(h \circ f \circ g) $ and in fact, each $ G_{k} $ and $ B_{k}$ for all $ k =1,2,3, \dots  $ is contained in the wandering component of  $ F(h \circ f \circ g) $. \\

\underline{Dynamical behavior of $h \circ g \circ f$}:\\	
\begin{equation} \label{cs33}
\left\{
\begin{array}{lll}
|(h \circ g \circ f)(z) -2|< 1/2,  \;\;\;& \forall z \in G_{0} \cup \big(\bigcup_{k =1}^{\infty} (L_{k}\cup  M_{k})\big);\\
|(h \circ g \circ f)(z) + 14|< 1/2, \;\;\;&   \forall z \in G_{k}, \; k =1,2,3,\ldots;\\
 |(h \circ g\circ  f)(z) + (4k +10)|< 1/2, \;\;\;&   \forall z \in B_{k}, \; k =1,2,3, \ldots; \\

\end{array}
\right\}
\end{equation}
The composition rule \ref{cs33} assigned above tells us  that domains $G_{0} \cup \big(\bigcup_{k =1}^{\infty} (L_{k}\cup  M_{k})\big)  $, $ G_{k} $ and $ B_{k}, \; (k =1,2,3,\ldots) $ lie in $ F(h \circ g \circ f) $ and in fact, each $ G_{k} $ and $ B_{k}$ for all $ k =1,2,3, \dots  $ is contained in the wandering component of  $ F(h \circ g \circ f) $. 

From all of the above discussion, we found that the domains $ G_{k} $ and $ B_{k} $ for all $ k =1,2,3, \ldots $ are contained in the wandering domains of the functions $ f, \; g, \; h $ and their composites.

\end{proof}


\begin{thebibliography}{30}
\bibitem {ber} Bergweiler, W. and Wang, Y.: \textit{On the dynamics of composite entire functions}, Ark. Mat. 36 (1998), 31-39.

\bibitem {gai} Gaier, A.: \textit{Lectures on complex approximation}, Birkhauser, 1987.

\bibitem {her} Herring, M. E.: \textit{Mapping properties of Fatou components}, Ann. Acad. Sci. Fenn. Math. 23 (1998), 263-274.

\bibitem {kum8} Kumar, D. and Kumar, S.: \textit{Dynamics of composite entire functions}, arXiv: 1207.5930v5[math.DS], 7 October, 2015. 

\bibitem {singh} Singh, A. P.: \textit{On the dynamics of composite entire functions}, Math. Proc. Camb. Phil. Soc. 134, (2003), 129-138.






 

\end{thebibliography}
\end{document}